\documentclass[12pt]{amsart}
\usepackage{amscd,amssymb,amsthm,verbatim}
\newcommand{\Flat}{\operatorname{flat}}

\newcommand{\Hess}{\operatorname{Hess}}

\newcommand{\Id}{\operatorname{Id}}

\newcommand{\R}{{\mathbb R}}

\newcommand{\Ric}{\operatorname{Ric}}

\numberwithin{equation}{section}
\theoremstyle{plain}
\newtheorem{definition}{Definition}
\newtheorem{assumption}{Assumption}
\newtheorem{lemma}{Lemma}
\newtheorem{theorem}{Theorem}
\newtheorem{proposition}{Proposition}

\theoremstyle{remark}

\newtheorem{remark}{Remark}
\newtheorem{example}{Example}

\begin{document}
\title{Mean curvature flow in a Ricci flow background}
\author{John Lott}
\address{Department of Mathematics\\
University of California - Berkeley\\
Berkeley, CA  94720-3840\\ 
USA} \email{lott@math.berkeley.edu}

\thanks{This research was partially supported
by NSF grant DMS-0903076}
\date{January 31, 2012}

\begin{abstract}
Following work of Ecker \cite{Ecker (2007)}, we consider
a weighted Gibbons-Hawking-York functional on a Riemannian
manifold-with-boundary.  We compute its variational
properties and its time derivative under Perelman's modified
Ricci flow. The answer has a boundary term which involves an extension of 
Hamilton's differential
Harnack expression for the mean curvature flow in Euclidean
space. We also derive the evolution equations for the
second fundamental form and the mean curvature, under a
mean curvature flow in a Ricci flow background.
In the case of a gradient Ricci soliton background, we
discuss mean curvature solitons and Huisken monotonicity.   
\end{abstract}

\maketitle


\section{Introduction} \label{section1}

In \cite{Ecker (2007)}, Ecker found a surprising link between
Perelman's ${\mathcal W}$-functional for Ricci flow and Hamilton's 
differential Harnack
expression for mean curvature flow in $\R^n$. 
If $\Omega \subset \R^n$ is a bounded domain with smooth boundary,
he considered the integral over $\Omega$ of Perelman's 
${\mathcal W}$-integrand
\cite[Proposition 9.1]{Perelman (2002)},
the latter being defined using a positive solution $u$ of the
backward heat equation. With an appropriate boundary condition
on $u$, the time-derivative of the integral has two
terms. The first term is the integral over $\Omega$ of a
nonnegative quantity, as in Perelman's work. The second
term is an integral over $\partial \Omega$. Ecker
showed that the integrand of the second term is 
Hamilton's differerential Harnack expression for mean curvature
flow \cite{Hamilton (1995)}. 
Hamilton had proven that this expression is nonnegative
for a weakly convex mean curvature flow in $\R^n$.

After performing diffeomorphisms generated by $\nabla \ln u$,
the boundary of $\Omega$ evolves by mean curvature
flow in $\R^n$. Ecker conjectured that his
${\mathcal W}$-functional for $\Omega$ is nondecreasing in time
under the mean curvature flow of any compact
hypersurface in $\R^n$. This conjecture is still open.

In the present paper we look at analogous relations for
mean curvature flow in
an arbitrary Ricci flow background. We begin with a
version of Perelman's ${\mathcal F}$-functional
\cite[Section 1.1]{Perelman (2002)}
for a manifold-with-boundary $M$. We add a boundary
term to the interior integral of ${\mathcal F}$ 
so that the result $I_\infty$ has
nicer variational properties.  One can think of $I_\infty$
as a weighted version of the Gibbons-Hawking-York functional 
\cite{Gibbons-Hawking (1977),York (1972)}, where
``weighted'' refers to a measure $e^{-f} \: dV_g$. We compute
how $I_\infty$ changes under a variation of the Riemannian
metric $g$ (Proposition \ref{prop2}). 
We also compute the time-derivative of $I_\infty$
when $g$ evolves by Perelman's modified Ricci flow
(Theorem \ref{thm2}).
We derive the evolution equations for the second fundamental
form of $\partial M$ and the mean curvature 
of $\partial M$ under the modified Ricci flow
(Theorem \ref{thm3}).

After performing diffeomorphisms generated by
$- \nabla f$, the Riemannian metric on $M$ evolves by the 
standard Ricci flow and $\partial M$ evolves by 
mean curvature flow. 

\begin{theorem} \label{thm1}
If $u = e^{-f}$ is a solution to the conjugate heat equation
\begin{equation} \label{1.1}
\frac{\partial u}{\partial t} = (- \triangle + R) u
\end{equation}
on $M$, satisfying the boundary condition
\begin{equation}
H + e_0 f = 0
\end{equation}
on $\partial M$,
then
\begin{align} \label{1.2}
& \frac{dI_\infty}{dt} =  2 \int_{M} |\Ric + \Hess(f)|^2 \: e^{-f} \: dV \\
& + \: 2 \int_{\partial M}  
\left( \frac{\partial H}{\partial t} -
2 \langle \widehat{\nabla} f, \widehat{\nabla} H \rangle
+ A( \widehat{\nabla} f, \widehat{\nabla} f)
+ 2 \Ric(e_0,\widehat{\nabla} f) - \frac12 e_0 R - H \Ric(e_0, e_0)
\right)
\: e^{-f} \: dA. \notag 
\end{align}
\end{theorem}

Here $R$ is the scalar curvature of $M$,
$\widehat{\nabla}$ is the boundarywise derivative,
$e_0$ is the inward unit normal on $\partial M$, 
$H$ is the mean curvature of
$\partial M$ 
and $A(\cdot, \cdot)$ is the second fundamental form of $\partial M$.

\begin{remark}
If $g(t)$ is flat Ricci flow on $\R^n$ then the boundary integrand
\begin{equation} \label{1.3}
\frac{\partial H}{\partial t} -
2 \langle \widehat{\nabla} f, \widehat{\nabla} H \rangle
+ A( \widehat{\nabla} f, \widehat{\nabla} f)
+ 2 \Ric(e_0, \widehat{\nabla} f) - \frac12 e_0 R - H \Ric(e_0, e_0)
\end{equation}
becomes Hamilton's differential Harnack expression
\cite[Definition 4.1]{Hamilton (1995)}
\begin{equation} \label{Harnack}
Z = \frac{\partial H}{\partial t} +
2 \langle V, \widehat{\nabla} H \rangle
+ A(V, V)
\end{equation}
when the vector field
$V$ of (\ref{Harnack}) is taken to be
$- \widehat{\nabla} f$. In this flat case, Theorem \ref{thm1} is the
${\mathcal F}$-version of Ecker's result.
\end{remark}

\begin{remark}
Writing (\ref{1.3}) as
\begin{equation}
\frac{\partial H}{\partial t} +
2 \langle V, \widehat{\nabla} H \rangle - H \Ric(e_0, e_0)
+ \left\langle e_0,
\nabla_V V - \frac12 \nabla R 
- 2 \Ric(V, \cdot) \right\rangle
\end{equation}
(with $V = - \widehat{\nabla} f$)
indicates a link to ${\mathcal L}$-geodesics, since
$\nabla_V V - \frac12 \nabla R 
- 2 \Ric(V, \cdot)$ (with $V = \gamma^\prime$) enters in the Euler-Lagrange
equation for the steady version of Perelman's ${\mathcal L}$-length
\cite[Section 4]{Lott (2009)}.
\end{remark}

Important examples of Ricci flow solutions come from
gradient solitons.  With such a background geometry, 
there is a natural notion
of a mean curvature soliton. We show that (\ref{1.3}) vanishes
on such solitons (Proposition \ref{prop6}), in analogy to what happens for
mean curvature flow in $\R^n$
\cite[Lemma 3.2]{Hamilton (1995)}. 
On the other hand,
in the case of convex mean curvature flow
in $\R^n$, Hamilton showed that the shrinker version of
(\ref{Harnack}) is nonnegative
for all vector fields $V$ \cite[Theorem 1.1]{Hamilton (1995)}. 
We cannot expect a direct analog
for mean curvature flow in an arbitrary gradient shrinking soliton background,
since local convexity of the hypersurface will generally not be
preserved by the flow.

Magni-Mantegazza-Tsatis
\cite{Magni-Mantegazza-Tsatis (2009)} showed that Huisken's monotonicity
formula for mean curvature flow in $\R^n$ 
\cite[Theorem 3.1]{Huisken (1990)} 
extends to mean curvature flow in a gradient Ricci soliton background.
We give two proofs of this result
(Proposition \ref{prop7}).  The first one is computational
and is essentially the same as the proof in 
\cite{Magni-Mantegazza-Tsatis (2009)}; the second one is
more conceptual.

In this paper we mostly deal with ``steady'' quantities :
${\mathcal F}$-functional, gradient steady soliton, etc.
There is an evident extension to the shrinking or expanding
case.

The structure of the paper is as follows. Section \ref{section2}
has some preliminary material. In Section
\ref{section3} we define the weighted Gibbons-Hawking-York action and study
its variational properties. Section \ref{section4} contains the derivation of
the evolution equations for the second fundamental form and
the mean curvature of the boundary, when the Riemannian metric
of the interior evolves by the modified Ricci flow. In Section
\ref{section5} we consider mean curvature flow in a Ricci flow background
and prove Theorem \ref{thm1}.
In the case of a gradient steady Ricci soliton background, we
discuss mean curvature solitons and Huisken monotonicity.

More detailed descriptions appear at the beginnings of the sections.

\section{Preliminaries} \label{section2}

In this section we gather some useful formulas about the geometry
of a Riemannian manifold-with-boundary. We will use the Einstein
summation convention freely.

Let $M$ be a smooth compact $n$-dimensional manifold-with-boundary.
We denote local coordinates for $M$ by $\{x^\alpha\}_{\alpha=0}^{n-1}$.
Near $\partial M$, we take $x^0$ to be a local defining function
for $\partial M$. We denote the local coordinates for
$\partial M$ by $\{x^i\}_{i=1}^n$.

If $g$ is a Riemannian metric on $M$ then we let $\nabla$ denote
the Levi-Civita connection on $TM$ and we let $\widehat{\nabla}$
denote the Levi-Civita connection on $T\partial M$. Algorithmically,
when taking covariant derivatives we use $\Gamma^\alpha_{\: \: \beta \gamma}$ to
act on indices in $\{0, \ldots, n-1\}$ and $\widehat{\Gamma}^i_{\: \: jk}$ to
act on indices in $\{1, \ldots, n-1\}$. We let $dV$ denote the
volume density on $M$ and we let $dA$ denote the area density on 
$\partial M$.

Let $e_0$ denote the inward-pointing unit
normal field on $\partial M$. For calculations, we choose local
coordinates near a point of $\partial M$
so that $\partial_0 \big|_{\partial M} = e_0$
and for all $(x^1, \ldots, x^{n-1})$, the curve 
$t \rightarrow (t, x^1, \ldots, x^{n-1})$ is a
unit-speed geodesic which meets $\partial M$ orthogonally.
In these coordinates, we can write
\begin{equation} \label{2.1}
g = (dx^0)^2 + \sum_{i,j=1}^{n-1} g_{ij}(x^0, x^1 \ldots, x^{n-1}) \:
dx^i \: dx^j.
\end{equation}
We write $A = (A_{ij})$ for the second fundamental form
of $\partial M$,
so $A_{ij} = g(e_0, \nabla_{\partial_j} \partial_i)$, and we write
$H = g^{ij} A_{ij}$ for the mean curvature. Then on $\partial M$,
we have
\begin{equation} \label{2.2}
A_{ij} = \Gamma_{0ij} = - \Gamma_{i0j} = - \Gamma_{ij0} = 
- \frac12 \partial_0 g_{ij}.
\end{equation}

The Codazzi-Mainardi equation
\begin{equation} \label{2.3}
R_{0ijk} = \widehat{\nabla}_j A_{ik} - \widehat{\nabla}_k A_{ij}
\end{equation}
implies that
\begin{equation} \label{2.4}
R_{0j} = \widehat{\nabla}_j H - \widehat{\nabla}_i A^i_{\: \: j}. 
\end{equation}

For later reference,
\begin{equation} \label{2.5}
\nabla_i R_{0j} = \widehat{\nabla}_i R_{0j} - \Gamma^k_{\: \: 0i} R_{kj}
- \Gamma^0_{\: \: ji} R_{00} = \widehat{\nabla}_i R_{0j} + A^k_{\: \: i} R_{kj}
- A_{ij} R_{00}.
\end{equation}
For any symmetric $2$-tensor field $v$, we have
\begin{equation} \label{2.6}
\nabla_0 \left( g^{ij} v_{ij} \right) = 
g^{ij} \partial_0 v_{ij} + 2 A^{ij} v_{ij} = 
g^{ij} \nabla_0 v_{ij}
\end{equation}
on $\partial M$.
More generally, on $\partial M$, $g^{ij}$ is covariantly constant
with respect to $\nabla_0$.

If $f \in C^\infty(M)$ then on $\partial M$, we have
\begin{equation} \label{2.7}
\nabla_i \nabla_j f = \widehat{\nabla}_i \widehat{\nabla}_j f
- \Gamma^0_{\: \: ji} \nabla_0 f = \widehat{\nabla}_i \widehat{\nabla}_j f
- A_{ij} \nabla_0 f
\end{equation}
and
\begin{equation} \label{2.8}
\nabla_i \nabla_0 f = \widehat{\nabla}_i \nabla_0 f
- \Gamma^k_{\: \: 0i} \widehat{\nabla}_k f = \widehat{\nabla}_i \nabla_0 f
+ A_i^{\: \: k} \: \widehat{\nabla}_k f
\end{equation}

\section{Variation of the weighted Gibbons-Hawking-York action}
\label{section3}

In this section we define the weighted Gibbons-Hawking-York
action $I_\infty$ and study its variational properties.

In Subsection \ref{subsection3.1} we list how some geometric quantities vary
under changes of the metric.  As a warmup, in Subsection \ref{subsection3.2}
we rederive the variational formula for the Gibbons-Hawking action.
In Subsection \ref{subsection3.3}
we derive the variational formula for the weighted Gibbons-Hawking action.
In Subsection \ref{subsection3.4} we compute its time derivative under the
modified Ricci flow.

\subsection{Variations} \label{subsection3.1}

Let $\delta g_{\alpha \beta} = v_{\alpha \beta}$ be a variation of $g$.
We write $v = g^{\alpha \beta} v_{\alpha \beta}$. We collect some
variational equations :
\begin{equation} \label{3.1}
\delta R = \nabla_\alpha \nabla_\beta v^{\alpha \beta} - \nabla_\alpha
\nabla^\alpha v - v^{\alpha \beta} R_{\alpha \beta},
\end{equation}

\begin{equation} \label{3.2}
\delta (dV) = \frac{v}{2} \: dV,
\end{equation}

\begin{equation} \label{3.3}
\delta(e_0) = - \frac12 v_0^{\: \: 0} e_0 - v_0^{\: \: k} \partial_k,
\end{equation}

\begin{align} \label{3.4}
\delta A_{ij} & = \frac12 \left( \nabla_i v_{0j} + \nabla_j v_{0i}
- \nabla_0 v_{ij} + A_{ij} v_{00} \right) \\
& = \frac12 \left( \widehat{\nabla}_i v_{0j} + A_{ki} v^k_{\: \: j}
+ \widehat{\nabla}_j v_{0i} + A_{kj} v^k_{\: \: i}
- \nabla_0 v_{ij} - A_{ij} v_{00} \right), \notag
\end{align}

\begin{equation} \label{3.5}
\delta H = - v^{ij} A_{ij} + g^{ij} \delta A_{ij} =
\widehat{\nabla}_i v_0^{\: \: i} - \frac12 \left(
g^{ij} \nabla_0 v_{ij} + H v_{00} \right), 
\end{equation}

\begin{equation} \label{3.6}
\delta (dA) = \frac12 \: v_i^{\: \: i} \: dA,
\end{equation}

\subsection{Gibbons-Hawking-York action} \label{subsection3.2}

\begin{definition} \label{def1}
The Gibbons-Hawking-York action 
\cite{Gibbons-Hawking (1977),York (1972)}
is
\begin{equation} \label{3.7}
I_{GHY}(g) = \int_M R \: dV \: + \: 2 \: \int_{\partial M} H \: dA.
\end{equation}
\end{definition}

If $n = 2$ then $I_{GHY}(g) = 4\pi \chi(M)$.

\begin{proposition} \label{prop1}
\begin{equation} \label{3.8}
\delta I_{GHY} = - \int_M v^{\alpha \beta} \left( R_{\alpha \beta} -
\frac12 R g_{\alpha \beta} \right) \: dV \: - \:
\int_{\partial M} v^{ij} \left( A_{ij} - g_{ij} H \right) \: dA. 
\end{equation}
\end{proposition}
\begin{proof}
From (\ref{3.1}) and (\ref{3.2}),
\begin{equation} \label{3.9}
\delta \int_M R \: dV \: = \:
- \int_M v^{\alpha \beta} \left( R_{\alpha \beta} -
\frac12 R g_{\alpha \beta} \right) \: dV \: - \:
\int_{\partial M} \left( \nabla_\alpha v_0^{\: \: \alpha} - \nabla_0 v
\right) \: dA.
\end{equation}
On the boundary,
\begin{align} \label{3.10}
\nabla_\alpha v_0^{\: \: \alpha} - \nabla_0 v = &
\nabla_i v_0^{\: \: i}  - \nabla_0 v_i^{\: \: i} \\
= & \widehat{\nabla}_i v_0^{\: \: i} - 
\Gamma^j_{\: \: 0i} \: v_j^{\: \: i} + 
\Gamma^i_{\: \: 0i} \: v_0^{\: \: 0}
- \nabla_0 (g^{ij} v_{ij}) \notag \\
= & \widehat{\nabla}_i v_0^{\: \: i} + A^{ij} v_{ij} - H v_0^{\: \: 0}
- g^{ij} \nabla_0 v_{ij}. \notag
\end{align}

From (\ref{3.5}) and (\ref{3.6}),
\begin{equation} \label{3.11}
\delta(H \: dA) = \left( \widehat{\nabla}_i v_0^{\: \: i} + \frac12 \left(
- g^{ij} \nabla_0 v_{ij} - H v_{00} + H v_i^{\: \: i} \right) \right)
\: dA.
\end{equation}

Combining (\ref{3.9}), (\ref{3.10}) and (\ref{3.11}) gives
\begin{align} \label{3.12}
\delta I_{GHY} & = - \int_M v^{\alpha \beta} \left( R_{\alpha \beta} -
\frac12 R g_{\alpha \beta} \right) \: dV \: 
+ \: \int_{\partial M} \widehat{\nabla}_i v_0^{\: \: i} \: dA \:
- \:
\int_{\partial M} v^{ij} \left( A_{ij} - g_{ij} H \right) \: dA \\
& = - \int_M v^{\alpha \beta} \left( R_{\alpha \beta} -
\frac12 R g_{\alpha \beta} \right) \: dV \: - \:
\int_{\partial M} v^{ij} \left( A_{ij} - g_{ij} H \right) \: dA. \notag
\end{align}
This proves the proposition.
\end{proof}

If the induced metric $g_{\partial M}$ is held
fixed under the variation then $v_{ij}$ vanishes on $\partial M$ and
$\delta I_{GHY} = - \int_M v^{\alpha \beta} \left( R_{\alpha \beta}
- \frac12 R g_{\alpha \beta} \right) \: dV$ is an interior integral.
This was the motivation for Gibbons and Hawking to introduce
the second term on the right-hand side of (\ref{3.7}).

Suppose that $n > 2$.
We can say that with a fixed induced metric $g_{\partial M}$ on $\partial M$,
the critical points of $I_{GHY}$ are the Ricci-flat metrics on $M$
that induce $g_{\partial M}$. On the other hand, if we
consider all variations $v_{\alpha \beta}$ then the critical
points are the Ricci-flat metrics on $M$ with totally geodesic boundary.

\subsection{Weighted Gibbons-Hawking-York action} \label{subsection3.3}

Given $f \in C^\infty(M)$, consider the smooth metric-measure space
${\mathcal M} = \left( M, g, e^{-f} dV \right)$. As is now well understood,
the analog of the Ricci curvature for ${\mathcal M}$ is the
Bakry-Emery tensor 
\begin{equation} \label{3.13}
\Ric_\infty = \Ric + \Hess(f).
\end{equation}
(There is also
a notion of $\Ric_N$ for $N \in [1, \infty]$ but we only consider
the case $N = \infty$.) As explained by Perelman
\cite[Section 1.3]{Perelman (2002)}, the analog of the scalar curvature
is 
\begin{equation} \label{3.14}
R_\infty = R + 2 \triangle f - |\nabla f|^2.
\end{equation}

Considering the first variation formula for the integral of $e^{-f}$ over a
moving hypersurface, one sees that the analog of the mean curvature is
\begin{equation} \label{3.15}
H_\infty = H + e_0 f.
\end{equation} 
On the other hand, the analog of the
second fundamental form is just $A_\infty = A$.

If $\partial M = \emptyset$ then Perelman's ${\mathcal F}$-functional is
the weighted total scalar curvature 
${\mathcal F} = \int_M R_\infty \: e^{-f} \: dV$.

\begin{definition} \label{def2}
The weighted Gibbons-Hawking-York action is
\begin{equation} \label{3.16}
I_\infty(g,f) = \int_M R_\infty \: e^{-f} \: dV \: + \: 2 \:
\int_{\partial M} H_\infty \: e^{-f} \: dA.
\end{equation}
\end{definition}

We write a variation of $f$ as $\delta f = h$.
Note that
\begin{equation} \label{3.17}
\delta \left( e^{-f} \: dV \right) = \left( \frac{v}{2} - h \right) \:
e^{-f} \: dV.
\end{equation}

\begin{proposition} \label{prop2}
If $\frac{v}{2} - h = 0$ then
\begin{equation} \label{3.18}
\delta I_{\infty} = - \int_M v^{\alpha \beta} \left( R_{\alpha \beta} +
\nabla_\alpha \nabla_\beta f \right) \: e^{-f} \: dV \: - \:
\int_{\partial M} \left( v^{ij} A_{ij} + v^{00} (H + e_0 f) 
\right) \: e^{-f} \: dA. 
\end{equation}
\end{proposition}
\begin{proof}
One can check that
\begin{align} \label{3.19}
\delta(\triangle f) & = \triangle h - 
\left( \nabla_\alpha v^{\alpha \beta} \right) \nabla_\beta f -
v^{\alpha \beta} \nabla_\alpha \nabla_\beta f + \frac12 \langle
\nabla f,
\nabla v \rangle \\
& = \frac12 \triangle v - 
\left( \nabla_\alpha v^{\alpha \beta} \right) \nabla_\beta f -
v^{\alpha \beta} \nabla_\alpha \nabla_\beta f + \frac12 \langle
\nabla f,
\nabla v \rangle \notag
\end{align}
and
\begin{equation} \label{3.20}
\delta \left(|\nabla f|^2 \right) = 2 \langle \nabla f, \nabla h \rangle
- v^{\alpha \beta} \nabla_\alpha f \nabla_\beta f
=  \langle \nabla f, \nabla v \rangle
- v^{\alpha \beta} \nabla_\alpha f \nabla_\beta f.
\end{equation}
Then
\begin{equation} \label{3.21}
\delta R_\infty = - v^{\alpha \beta} \left( R_{\alpha \beta} +
\nabla_\alpha \nabla_\beta f \right) \: e^{-f} \: 
+ \: \nabla_\beta \left( e^{-f} \left( \nabla_\alpha v^{\beta \alpha}
- v^{\beta \alpha} \nabla_\alpha f \right) \right).
\end{equation}
Hence
\begin{align} \label{3.22}
\delta \left( \int_M R_\infty \: e^{-f} \: dV \right)  = &
\int_M \delta (R_\infty) \: e^{-f} \: dV \\
= & - \int_M v^{\alpha \beta} \left( R_{\alpha \beta} + \nabla_\alpha
\nabla_\beta f \right) \: e^{-f} \: dV \: \notag \\
&- \int_{\partial M} \left( \nabla_\alpha v^{0\alpha} - v^{0 \alpha}
\nabla_\alpha f \right) \: e^{-f} \: dA. \notag
\end{align}
On the boundary,
\begin{equation} \label{3.23}
\nabla_\alpha v^{0\alpha} - v^{0 \alpha}
\nabla_\alpha f = \nabla_0 v^{00} - v^{00} (H + \nabla_0 f) + 
\widehat{\nabla}_i v^{0i} - v^{0i} \widehat{\nabla}_i f 
+ v^{ij} A_{ij}.
\end{equation}

Next, 
\begin{equation} \label{3.24}
\delta (e_0 f) = - \frac12 v_0^{\: \: 0} \: \nabla_0 f - v_0^{\: \: i} \:
\widehat{\nabla}_i f
+ \nabla_0 h = - \frac12 v_0^{\: \: 0} \: \nabla_0 f - v_0^{\: \: i} \: 
\widehat{\nabla}_i f
+ \frac12 \nabla_0 v
\end{equation}
and one finds that
\begin{align} \label{3.25}
\delta \int_{\partial M} H_\infty \: e^{-f} \: dA = &
\int_{\partial M} \delta H_\infty \: e^{-f} \: dA +
\int_{\partial M} H_\infty \: \left( - \frac{1}{2} v + \frac12 v_i^{\: \: i}
\right) \: e^{-f} \: dA  \\
= & \int_{\partial M} \left( \widehat{\nabla}_i v^{0i} - v^{0i} 
\widehat{\nabla}_i f - v^{00} (H + e_0 f) + \frac12 \nabla_0 v^{00}
\right) \: e^{-f} \: dA. \notag
\end{align}

Combining (\ref{3.22}), (\ref{3.23}) and (\ref{3.25}) gives
\begin{align} \label{3.26}
\delta I_\infty = & - \int_M v^{\alpha \beta} 
\left( R_{\alpha \beta} + \nabla_\alpha
\nabla_\beta f \right) \: e^{-f} \: dV \: 
- \: \int_{\partial M} 
\left( v^{ij} A_{ij} + v^{00} (H + e_0 f) \right)
\: e^{-f} \: dA \\
& + \: \int_{\partial M} \left( \widehat{\nabla}_i v^{0i} - v^{0i} 
\widehat{\nabla}_i f \right) \: e^{-f} \: dA \notag \\
= & - \int_M v^{\alpha \beta} 
\left( R_{\alpha \beta} + \nabla_\alpha
\nabla_\beta f \right) \: e^{-f} \: dV \: 
- \: \int_{\partial M} 
\left( v^{ij} A_{ij} + v^{00} (H + e_0 f) \right)
\: e^{-f} \: dA \notag \\
& + \: \int_{\partial M} \widehat{\nabla}_i \left( e^{-f} v^{0i} \right)
\: dA \notag \\
= & - \int_M v^{\alpha \beta} 
\left( R_{\alpha \beta} + \nabla_\alpha
\nabla_\beta f \right) \: e^{-f} \: dV \: 
- \: \int_{\partial M} 
\left( v^{ij} A_{ij} + v^{00} (H + e_0 f) \right)
\: e^{-f} \: dA. \notag
\end{align}
This proves the proposition.
\end{proof}

\begin{remark} \label{rmk1}
If $\partial M = \emptyset$ then Proposition \ref{prop2} appears in
\cite[Section 1.1]{Perelman (2002)}.
\end{remark}

The variations in Proposition \ref{prop2} all fix the measure
$e^{-f} \: dV$. If we also
fix an induced metric $g_{\partial M}$ on
$\partial M$ then the critical points of $I_\infty$ are
gradient steady solitons on $M$ that satisfy $H + e_0 f = 0$
on $\partial M$. On the other hand, if we allow variations
that do not fix the boundary metric then the critical points
are gradient steady solitons on $M$ with totally geodesic
boundary and for which $f$ satisfies Neumann boundary conditions.

\subsection{Time-derivative of the action} \label{subsection3.4}

\begin{assumption}
Hereafter we assume that $H + e_0 f = 0$ on $\partial M$.
\end{assumption}

Then on $\partial M$, we have
\begin{equation} \label{added1}
\nabla_i \nabla_j f = \widehat{\nabla}_i \widehat{\nabla}_j f +
H A_{ij} 
\end{equation}
and
\begin{equation} \label{3.27}
\nabla_i \nabla_0 f = - \widehat{\nabla}_i H + A_i^{\: \: k} \:
\widehat{\nabla}_k f.
\end{equation}

\begin{theorem} \label{thm2}
Under the assumptions 
\begin{equation} \label{3.28}
\frac{\partial g}{\partial t} = - 2 (\Ric + \Hess(f))
\end{equation}
and
\begin{equation} \label{3.29}
\frac{\partial f}{\partial t} = - \triangle f - R,
\end{equation}
we have
\begin{align} \label{3.30}
& \frac{dI_\infty}{dt} =  2 \int_{M} |\Ric + \Hess(f)|^2 \: e^{-f} \: dV \\
& + \: 2 \int_{\partial M}  
\left( \widehat{\triangle} H -
2 \langle \widehat{\nabla} f, \widehat{\nabla} H \rangle
+ A( \widehat{\nabla} f, \widehat{\nabla} f)
+ A^{ij} A_{ij} H + A^{ij} R_{ij} + 2 R^{0i}
\widehat{\nabla}_i f - \widehat{\nabla}_i R^{0i} \right)
\: e^{-f} \: dA. \notag 
\end{align}

If $\left( R_{ij} + \nabla_i \nabla_j f \right) \Big|_{\partial M} = 0$
and $\left( R_{i0} + \nabla_i \nabla_0 f \right) \Big|_{\partial M} = 0$ then
\begin{equation} \label{3.31}
\widehat{\triangle} H -
2 \langle \widehat{\nabla} f, \widehat{\nabla} H \rangle
+ A( \widehat{\nabla} f, \widehat{\nabla} f)
+ A^{ij} A_{ij} H + A^{ij} R_{ij} + 2 R^{0i}
\widehat{\nabla}_i f - \widehat{\nabla}_i R^{0i} = 0.
\end{equation}
\end{theorem}

\begin{proof}
Equations (\ref{3.28}) and (\ref{3.29}) imply that $e^{-f(t)} dV_{g(t)}$ is
constant in $t$.
Then Proposition \ref{prop2} implies that
\begin{equation} \label{3.32}
\frac{dI_\infty}{dt} = \: 2 \int_{M} |\Ric + \Hess(f)|^2 \: e^{-f} \: dV \:
+ \: 2 \int_{\partial M} A^{ij} \left( R_{ij} + {\nabla}_i {\nabla}_j f
\right) \: e^{-f} \: dA.
\end{equation}
 
\begin{lemma} \label{lem1}
On $\partial M$, 
\begin{align} \label{3.33}
& A^{ij} \left( R_{ij} + 
{\nabla}_i {\nabla}_j f \right) \: e^{-f} 
- \widehat{\nabla}_i \left( \left( R^{i0} + \nabla^i \nabla^0 f
\right) e^{-f} \right) =  \\
& \left( \widehat{\triangle} H -
2 \langle \widehat{\nabla} f, \widehat{\nabla} H \rangle
+ A( \widehat{\nabla} f, \widehat{\nabla} f)
+ A^{ij} A_{ij} H + A^{ij} R_{ij} + 2 R^{0i}
\widehat{\nabla}_i f - \widehat{\nabla}_i R^{0i} \right) \: e^{-f}. \notag
\end{align}
\end{lemma}
\begin{proof}
As
\begin{equation} \label{3.34}
A^{ij} (\widehat{\nabla}_i \widehat{\nabla}_j f) e^{-f} =
\widehat{\nabla}_i \left( A^{ij} (\widehat{\nabla}_j f) e^{-f} \right)
- \left(  \widehat{\nabla}_i  A^{ij} \right) \widehat{\nabla}_j f e^{-f}
+ A( \widehat{\nabla} f, \widehat{\nabla} f) e^{-f},
\end{equation}
we have
\begin{align} \label{3.35}
& A^{ij} \left(  R_{ij} + 
{\nabla}_i {\nabla}_j f
\right) \: e^{-f} = \\
& \left( A^{ij} R_{ij} + A^{ij} 
\widehat{\nabla}_i \widehat{\nabla}_j f + A^{ij} A_{ij} H
\right) \: e^{-f} = \notag \\
& \left( A^{ij} R_{ij} - \left( \widehat{\nabla}_i A^{ij} \right) 
\widehat{\nabla}_j f + A( \widehat{\nabla} f, \widehat{\nabla} f)
+ A^{ij} A_{ij} H \right) \: e^{-f} \notag \\
& + \widehat{\nabla}_i \left( A^{ij} (\widehat{\nabla}_j f) e^{-f} \right). 
\notag
\end{align}
Adding
\begin{equation} \label{3.36}
0 = \left( \widehat{\triangle} H -
\langle \widehat{\nabla} f, \widehat{\nabla} H \rangle \right) e^{-f} - 
\widehat{\nabla}_i \left( e^{-f}
\widehat{\nabla}^i H \right)
\end{equation}
and using (\ref{2.4}) gives
\begin{align} \label{3.37}
& A^{ij} \left( R_{ij} + 
{\nabla}_i {\nabla}_j f 
\right) \: e^{-f} = \\
& \left( \widehat{\triangle} H -
\langle \widehat{\nabla} f, \widehat{\nabla} H \rangle + 
A^{ij} R_{ij} - \left( \widehat{\nabla}_i A^{ij} \right) 
\widehat{\nabla}_j f + A( \widehat{\nabla} f, \widehat{\nabla} f)
+ A^{ij} A_{ij} H \right) \: e^{-f} \notag \\
& + \widehat{\nabla}_i \left( \left( A^{ij} \widehat{\nabla}_j f -
\widehat{\nabla}^i H \right) e^{-f} \right) = \notag \\
& \left( \widehat{\triangle} H -
2 \langle \widehat{\nabla} f, \widehat{\nabla} H \rangle + 
A^{ij} R_{ij} + R^{0j}
\widehat{\nabla}_j f + A( \widehat{\nabla} f, \widehat{\nabla} f)
+ A^{ij} A_{ij} H \right) \: e^{-f} \notag \\
& + \widehat{\nabla}_i \left( \left( A^{ij} \widehat{\nabla}_j f -
\widehat{\nabla}^i H \right) e^{-f} \right). \notag
\end{align}
Using (\ref{3.27}), 
\begin{align} \label{3.38}
\widehat{\nabla}_i \left( \left( A^{ij} \widehat{\nabla}_j f -
\widehat{\nabla}^i H \right) e^{-f} \right) = &
\widehat{\nabla}_i \left( \left( R^{i0} + \nabla^i \nabla^0 f
\right) e^{-f} \right)
- \widehat{\nabla}_i \left( R^{0i} e^{-f} \right) \\
= & \widehat{\nabla}_i \left( \left( R^{i0} + \nabla^i \nabla^0 f
\right) e^{-f} \right) \notag \\
& + 
\left( - \widehat{\nabla}_i R^{0i} + 
R^{0i} \widehat{\nabla}_i f \right) e^{-f}. \notag
\end{align}
The lemma follows.
\end{proof}

Equation (\ref{3.30}) follows from (\ref{3.32}), along with
integrating both sides of
(\ref{3.33}) over $\partial M$. 
If $\left( R_{ij} + \nabla_i \nabla_j f \right) \Big|_{\partial M} = 0$
and $\left( R_{i0} + \nabla_i \nabla_0 f \right) \Big|_{\partial M} = 0$ then
(\ref{3.31}) follows from (\ref{3.33}).
\end{proof}

\section{Evolution equations for the
boundary geometry under a modified Ricci flow}
\label{section4}

In this section we consider a manifold-with-boundary whose Riemannian metric
evolves by the modified Ricci flow.
We derive the evolution equations for the 
second fundamental form and the mean curvature of the boundary.

\begin{theorem} \label{thm3}
Under the assumptions 
\begin{equation} \label{4.1}
\frac{\partial g}{\partial t} = - 2 (\Ric + \Hess(f))
\end{equation}
and
\begin{equation} \label{4.2}
\frac{\partial f}{\partial t} = - \triangle f - R,
\end{equation}
on $\partial M$ we have
\begin{equation} \label{geqn}
\frac{\partial g_{ij}}{\partial t} = - 
({\mathcal L}_{\widehat{\nabla} f} g)_{ij}
- 2 R_{ij} - 2 H A_{ij},
\end{equation}
\begin{align} \label{4.3}
\frac{\partial A_{ij}}{\partial t} = & (\widehat{\triangle} A)_{ij}
- ({\mathcal L}_{\widehat{\nabla} f} A)_{ij} - A^k_{\: \: i} \:
R^l_{\: \: klj} -
A^k_{\: \: j} \: R^l_{\: \: kli} + 2 A^{kl} R_{kilj} \\
& - 2 H A_{ik} A^k_{\: \: j}
+ A^{kl} A_{kl} A_{ij} + \nabla_0 R_{0i0j} \notag
\end{align}
and
\begin{equation} \label{4.4}
\frac{\partial H}{\partial t} = \widehat{\triangle} H -
\langle \widehat{\nabla} f, \widehat{\nabla} H \rangle + 2 A^{ij} R_{ij} + 
A^{ij} A_{ij} H + \nabla_0 R_{00}.
\end{equation}
\end{theorem}
\begin{proof}
Using (\ref{added1}),
\begin{align}
\frac{\partial g_{ij}}{\partial t} = & - 2 R_{ij} - 2 \nabla_i \nabla_j f \\
= & - 2R_{ij} - 2 \widehat{\nabla}_i \widehat{\nabla}_j  f - 2 H A_{ij} 
\notag \\
= & - ({\mathcal L}_{\widehat{\nabla}f} g)_{ij} - 2R_{ij} - 2 H A_{ij}. 
\notag 
\end{align}

Next, from (\ref{3.4}),
\begin{align} \label{4.5}
\frac{\partial A_{ij}}{\partial t} = & - \nabla_i \left(
R_{j0} + \nabla_j \nabla_0 f \right) - \nabla_j \left(
R_{i0} + \nabla_i \nabla_0 f \right) + \nabla_0 \left(
R_{ij} + \nabla_i \nabla_j f \right) \\
& - A_{ij} \left( R_{00} + \nabla_0 \nabla_0 f \right). \notag
\end{align}
Now
\begin{equation} \label{4.6}
\nabla_0 \nabla_i \nabla_j f - \nabla_j \nabla_i \nabla_0 f =
\nabla_0 \nabla_j \nabla_i f - \nabla_j \nabla_0 \nabla_i f =
- R^k_{\: \: i0j} \widehat{\nabla}_k f - R^0_{\: \: i0j} \nabla_0 f 
\end{equation}
and
\begin{align} \label{4.7}
\nabla_i \nabla_j \nabla_0 f = & \widehat{\nabla}_i \nabla_j \nabla_0 f
- \Gamma^0_{\: \: ji} \nabla_0 \nabla_0 f - \Gamma^k_{\: \: 0i} 
\nabla_j \nabla_k f \\
= & \widehat{\nabla}_i \left( \widehat{\nabla}_j \nabla_0 f + A^k_{\: \: j}
\widehat{\nabla}_k f \right)
- A_{ij} \nabla_0 \nabla_0 f + A^k_{\: \: i} 
\left( \widehat{\nabla}_j \widehat{\nabla}_k f + H A_{jk} \right) 
\notag \\
= & - \widehat{\nabla}_i \widehat{\nabla}_j H + 
\left( \widehat{\nabla}_i A^k_{\: \: j} \right)
\widehat{\nabla}_k f + A^k_{\: \: j} \: \widehat{\nabla}_i
\widehat{\nabla}_k f 
- A_{ij} \nabla_0 \nabla_0 f \notag \\
& + A^k_{\: \: i} \:
\widehat{\nabla}_j \widehat{\nabla}_k f + H A^k_{\: \: i}  A_{jk}.
\notag  
\end{align}
Then
\begin{align} \label{4.8}
\frac{\partial A_{ij}}{\partial t} = & - \nabla_i R_{j0} - \nabla_j R_{i0} +
\nabla_0 R_{ij} - A_{ij} \left( R_{00} + \nabla_0 \nabla_0 f \right)
- R^k_{\: \: i0j} \widehat{\nabla}_k f + H R^0_{\: \: i0j} \\
& + \widehat{\nabla}_i \widehat{\nabla}_j H - 
\left( \widehat{\nabla}_i A^k_{\: \: j} \right)
\widehat{\nabla}_k f - A^k_{\: \: j} \: \widehat{\nabla}_i
\widehat{\nabla}_k f 
+ A_{ij} \nabla_0 \nabla_0 f \notag \\
& - A^k_{\: \: i} \:
\widehat{\nabla}_j \widehat{\nabla}_k f - H A^k_{\: \: i}  A_{jk} \notag \\
= & \: \widehat{\nabla}_i \widehat{\nabla}_j H - 
\left( \widehat{\nabla}_k A_{ij} \right)
\widehat{\nabla}_k f - A^k_{\: \: j} \: \widehat{\nabla}_i
\widehat{\nabla}_k f - A^k_{\: \: i} \: \widehat{\nabla}_j
\widehat{\nabla}_k f \notag \\
& - \nabla_i R_{j0} - \nabla_j R_{i0} + \nabla_0 R_{ij} 
- A_{ij} R_{00} + H R^0_{\: \: i0j} - H A^k_{\: \: i} A_{jk} \notag \\
= & \: \widehat{\nabla}_i \widehat{\nabla}_j H - 
\left( {\mathcal L}_{\widehat{\nabla} f} A \right)_{ij}
- \nabla_i R_{j0} - \nabla_j R_{i0} + \nabla_0 R_{ij} 
- A_{ij} R_{00} + H R^0_{\: \: i0j} - H A^k_{\: \: i} A_{jk} \notag
\end{align}
A form of Simons' identity \cite[Theorem 4.2.1]{Simons (1968)} says that
\begin{align} \label{4.9}
\widehat{\nabla}_i \widehat{\nabla}_j  H = & (\widehat{\triangle} A)_{ij}
+ \widehat{\nabla}_i R_{j0} + \widehat{\nabla}_j  R_{i0} -
\nabla_0 R_{ij} \\
& +  A^k_{\: \: i} R_{0k0j} + A^k_{\: \: j} R_{0k0i} - 
A_{ij} R_{00} + 2 A^{kl} R_{kilj} \notag \\
& - H R_{0i0j} - H A^k_{\: \: i} A_{jk}
+ A^{kl} A_{kl} A_{ij} + \nabla_0 R_{0i0j}. \notag
\end{align}
(As a check, one can easily show that the contraction of both sides of
(\ref{4.9}) with $g^{ij}$ is the same.)
Then using (\ref{2.5}),
\begin{align} \label{4.10}
\frac{\partial A_{ij}}{\partial t} = & (\widehat{\triangle} A)_{ij}
- \left( {\mathcal L}_{\widehat{\nabla} f} A \right)_{ij} - 
\left( \nabla_i R_{j0} - \widehat{\nabla}_i R_{j0} \right) - 
\left( \nabla_j R_{i0} - \widehat{\nabla}_j R_{i0} \right) \\
& - 2 A_{ij} R_{00}  + A^k_{\: \: i} R_{0k0j} + A^k_{\: \: j} R_{0k0i} 
+ 2 A^{kl} R_{kilj} \notag \\
& - 2 H A^k_{\: \: i} A_{jk}
+ A^{kl} A_{kl} A_{ij} + \nabla_0 R_{0i0j} \notag \\
= & (\widehat{\triangle} A)_{ij}
- ({\mathcal L}_{\widehat{\nabla} f} A)_{ij} - A^k_{\: \: i} R^l_{\: \: klj} -
A^k_{\: \: j} R^l_{\: \: kli} + 2 A^{kl} R_{kilj} \notag \\
& - 2 H A_{ik} A^k_{\: \: j}
+ A^{kl} A_{kl} A_{ij} + \nabla_0 R_{0i0j}. \notag
\end{align}
This proves the evolution equation for $A$.

Then
\begin{align} \label{4.11}
\frac{\partial H}{\partial t} = & \frac{\partial}{\partial t}
\left( g^{ij} A_{ij} \right) = 
2 \left( R_{ij} + \widehat{\nabla}_i \widehat{\nabla}_j f + H A_{ij} \right)
A^{ij} + g^{ij} \frac{\partial A_{ij}}{\partial t} \\
=  & 2 A^{ij} R_{ij} + \widehat{\triangle} H - 
\left( g^{ij} ( {\mathcal L}_{\widehat{\nabla} f} A)_{ij} 
- 2 A^{ij} \widehat{\nabla}_i
\widehat{\nabla}_j f \right) + A^{ij} A_{ij} H + g^{ij} 
\nabla_0 R_{0i0j} \notag \\
= & \widehat{\triangle} H -
\langle \widehat{\nabla} f, \widehat{\nabla} H \rangle + 2 A^{ij} R_{ij} + 
A^{ij} A_{ij} H + \nabla_0 R_{00}. \notag
\end{align}
This proves the theorem.
\end{proof}

\begin{proposition} \label{prop3}
Under the assumptions 
\begin{equation} \label{4.12}
\frac{\partial g}{\partial t} = - 2 (\Ric + \Hess(f))
\end{equation}
and
\begin{equation} \label{4.13}
\frac{\partial f}{\partial t} = - \triangle f - R,
\end{equation}
we have
\begin{align} \label{4.14}
& \frac{dI_\infty}{dt} =  2 \int_{M} |\Ric + \Hess(f)|^2 \: e^{-f} \: dV \\
& + \: 2 \int_{\partial M}  
\left( \frac{\partial H}{\partial t} -
\langle \widehat{\nabla} f, \widehat{\nabla} H \rangle
+ A( \widehat{\nabla} f, \widehat{\nabla} f)
+ 2 R^{0i} \widehat{\nabla}_i f - \frac12 \nabla_0 R - H R_{00}
\right)
\: e^{-f} \: dA. \notag 
\end{align}

If $\left( R_{ij} + \nabla_i \nabla_j f \right) \Big|_{\partial M} = 0$
and $\left( R_{i0} + \nabla_i \nabla_0 f \right) \Big|_{\partial M} = 0$ then
\begin{equation} \label{4.15}
\frac{\partial H}{\partial t} -
\langle \widehat{\nabla} f, \widehat{\nabla} H \rangle
+ A( \widehat{\nabla} f, \widehat{\nabla} f)
+ 2 R^{0i} \widehat{\nabla}_i f - \frac12 \nabla_0 R - H R_{00} = 0.
\end{equation}
\end{proposition}
\begin{proof}
From Theorem \ref{thm3},
\begin{align} \label{4.16}
& \widehat{\triangle} H -
2 \langle \widehat{\nabla} f, \widehat{\nabla} H \rangle
+ A( \widehat{\nabla} f, \widehat{\nabla} f)
+ A^{ij} A_{ij} H + A^{ij} R_{ij} + 2 R^{0i}
\widehat{\nabla}_i f - \widehat{\nabla}_i R^{0i} = \\
& \frac{\partial H}{\partial t} -
\langle \widehat{\nabla} f, \widehat{\nabla} H \rangle
+ A( \widehat{\nabla} f, \widehat{\nabla} f)
- A^{ij} R_{ij} + 2 R^{0i}
\widehat{\nabla}_i f - \widehat{\nabla}_i R^{0i} - \nabla_0 R^{00}. \notag 
\end{align}
From the second contracted Bianchi identity,
\begin{align} \label{4.17}
\frac12 \nabla^0 R = \nabla_i R^{i0} + \nabla_0 R^{00} =
\widehat{\nabla}_i R^{i0} + A^{ij} R_{ij} - H R^{00} + \nabla_0 R^{00}.
\end{align}
The proposition now follows from Theorem \ref{thm2}.
\end{proof}

\section{Hypersurfaces in a Ricci flow background} \label{section5}

In this section we consider mean curvature flow in a Ricci flow
background. Mean curvature flow in a fixed Riemannian manifold
was considered in \cite{Huisken (1986)}.

In Subsection \ref{subsection5.1} we translate the results of the previous
sections from a fixed manifold-with-boundary, equipped with a
modified Ricci flow, to an evolving hypersurface in a Ricci flow
solution.  

Starting in Subsection \ref{subsection5.2}, we consider mean curvature flow in
a gradient Ricci soliton background. We define what it means for a
hypersurface to be a mean curvature soliton.
We show that the differential 
Harnack-type expression vanishes on mean curvature
solitons.

In Subsection \ref{subsection5.3} we give two proofs of the monotonicity of
a Huisken-type functional. The first proof, which is calculational,
is essentially the same as the one in \cite{Magni-Mantegazza-Tsatis (2009)}.
The second proof is noncalculational.

\subsection{Mean curvature flow in a general Ricci flow background}
\label{subsection5.1}

Let $M$ be a smooth $n$-dimensional manifold and let $g(\cdot)$ satisfy the
Ricci flow equation $\frac{\partial g}{\partial t} = - 2 \Ric$. Given an
$(n-1)$-dimensional manifold $\Sigma$, let $\{e(\cdot)\}$ be a smooth
one-parameter family of immersions of $\Sigma$ in $M$. We write
$\Sigma_t$ for the image of $\Sigma$ under $e(t)$, and consider
$\{\Sigma_t\}$ to be a $1$-parameter family of parametrized 
hypersurfaces in $M$. Suppose that $\{\Sigma_t\}$ evolves by the
mean curvature flow 
\begin{equation}
\frac{dx}{dt} = H e_0.
\end{equation}

\begin{proposition} \label{prop4}
We have
\begin{equation} \label{geqn2}
\frac{\partial g_{ij}}{\partial t} = - 2 R_{ij} - 2 H A_{ij},
\end{equation}
\begin{align} \label{5.1}
\frac{\partial A_{ij}}{\partial t} = & (\widehat{\triangle} A)_{ij}
- A^k_{\: \: i} R^l_{\: \: klj} -
A^k_{\: \: j} R^l_{\: \: kli} + 2 A^{kl} R_{kilj} \\
& - 2 H A_{ik} A^k_{\: \: j}
+ A^{kl} A_{kl} A_{ij} + \nabla_0 R_{0i0j} \notag
\end{align}
and 
\begin{equation} \label{5.2}
\frac{\partial H}{\partial t} = \widehat{\triangle} H + 2 A^{ij} R_{ij} + 
A^{ij} A_{ij} H + \nabla_0 R_{00}.
\end{equation}
\end{proposition}
\begin{proof}
Suppose first that $\Sigma_t = \partial X_t$ with each $X_t$ compact.
Given a time interval $[a,b]$, find a positive solution 
on $\bigcup_{t \in [a,b]} (X_t \times \{t\}) \subset M \times [a,b]$
of the
conjugate heat equation
\begin{equation} \label{5.3}
\frac{\partial u}{\partial t} = (- \triangle + R) u,
\end{equation}
satisfying the boundary condition $e_0 u = Hu$, 
by solving it backwards in time from $t=b$.
(Choosing diffeomorphisms from $\{X_t\}$ to $X_a$, we can
reduce the problem of solving (\ref{5.3}) to a parabolic equation
on a fixed domain.)
Define $f$ by $u = e^{- f}$.

Let $\{\phi_t\}_{t \in [a,b]}$ be the one-parameter family of
diffeomorphisms generated by $\{- \nabla_{g(t)} f(t)\}_{t \in [a,b]}$, with
$\phi_a = \Id$. Then
$\phi_t(X_a) = X_t$ for all $t$.
Put $\widehat{g}(t) = \phi_t^* g(t)$ and $\widehat{f}(t) = \phi_t^* f(t)$.
Then
\begin{itemize}
\item $\widehat{g}(t)$ and $\widehat{f}(t)$ are defined on $X_a$, \\
\item $\frac{\partial \widehat{g}}{\partial t} = - 2 (\Ric_{\widehat{g}} + 
\Hess( \widehat{f} ))$, \\
\item $e_0 \widehat{f} + \widehat{H} = 0$ and
\item the measure 
$e^{- \widehat{f}(t)} \: dV_{\widehat{g}(t)}$
is constant in $t$.
\end{itemize} 
The proposition now follows from applying 
$(\phi_t^*)^{-1}$ to
equations (\ref{geqn}), (\ref{4.3}) and (\ref{4.4}) (the latter three being
written in terms of $\widehat{g}$ and $\widehat{f}$).

As the result could be derived from a local calculation
on $\Sigma_t$, it is also valid
without the assumption that $\Sigma_t$ bounds a compact domain.
\end{proof}

\begin{example} \label{ex1}
If $M = \R^n$ and $g(t) = g_{\Flat}$ then
equations (\ref{geqn2}), (\ref{5.1}) and (\ref{5.2}) are the same as 
\cite[Lemma 3.2, Theorem 3.4 and Corollary 3.5]{Huisken (1984)}
\end{example}

\begin{proposition} \label{prop5}
If $u = e^{-f}$ is a solution to the conjugate heat equation
(\ref{5.3}) then
\begin{align} \label{5.4}
& \frac{dI_\infty}{dt} =  2 \int_{M} |\Ric + \Hess(f)|^2 \: e^{-f} \: dV \\
& + \: 2 \int_{\partial M}  
\left( \frac{\partial H}{\partial t} -
2 \langle \widehat{\nabla} f, \widehat{\nabla} H \rangle
+ A( \widehat{\nabla} f, \widehat{\nabla} f)
+ 2 R^{0i} \widehat{\nabla}_i f - \frac12 \nabla_0 R - H R_{00}
\right)
\: e^{-f} \: dA. \notag 
\end{align}
\end{proposition}
\begin{proof}
This follows from Proposition \ref{prop3}.
\end{proof}

\begin{example} \label{ex2}
If $M = \R^n$, and $g(t) = g_{\Flat}$ then
Proposition \ref{prop5} is the same as 
\cite[Propositions 3.2 and 3.4]{Ecker (2007)}, after making the
change from the ${\mathcal F}$-type functional considered in this
paper to the ${\mathcal W}$-type functional considered in
\cite{Ecker (2007)}.
\end{example}

We will need the next lemma later.

\begin{lemma} \label{area}
We have
\begin{equation}
\frac{d}{dt} (dA) = - \left( R_i^{\: \: i} + H^2 \right) \: dA.
\end{equation}
\end{lemma}
\begin{proof}
Using (\ref{geqn2}),
\begin{equation}
\frac{d}{dt} (dA) = \frac12 \left( g^{ij} \: \frac{\partial g_{ij}}{\partial t}
\right) \: dA =
- \left( R_i^{\: \: i} + H^2 \right) \: dA.
\end{equation}
This proves the lemma.
\end{proof}

Using Lemma \ref{area}, we prove that a
mean curvature flow of two-spheres, in a three-dimensional immortal
Ricci flow solution, must have a finite-time singularity.

\begin{proposition} \label{ext}
Suppose that $(M, g(\cdot))$ is a three-dimensional Ricci flow
solution that is defined for $t \in [0, \infty)$,
with complete time slices and uniformly
bounded curvature on compact time intervals.
If $\{\Sigma_t\}$ is a mean curvature flow of two-spheres in
$(M, g(\cdot))$ then
the mean curvature flow has a finite-time singularity.
\end{proposition}
\begin{proof}
We estimate the area of $\Sigma_t$, along the lines of Hamilton's
area estimate for minimal disks in a Ricci flow solution
\cite[Section 11]{Hamilton (1999)},
\cite[Lemma 91.12]{Kleiner-Lott (2008)}. Let ${\mathcal A}(t)$ 
denote the area of
$\Sigma_t$.
From Lemma \ref{area},
\begin{equation}
\frac{d{\mathcal A}}{dt} = - \int_{\Sigma_t} (R_i^{\: \: i} + H^2) \: dA.
\end{equation}
Now
\begin{equation}
R_i^{\: \: i} = \frac12 \left( R + R_{ij}^{\: \: \: \: ij} \right) =
\frac12 \left( R + \widehat{R} - H^2 + A^{ij} A_{ij} \right),
\end{equation}
where $\widehat{R}$ denotes the scalar curvature of $\Sigma_t$.
From a standard Ricci flow estimate \cite[(B.2)]{Kleiner-Lott (2008)},
\begin{equation}
R(x,t) \ge - \frac{3}{2t}.
\end{equation}
Then
\begin{equation}
\frac{d{\mathcal A}}{dt} \: = \: - \: \frac12 \int_{\Sigma_t}
\left( R + \widehat{R} + H^2 + A^{ij} A_{ij} \right) \: dA \: \le
\frac{3}{4t} {\mathcal A}(t) - 2 \pi \chi(\Sigma_t) = 
\frac{3}{4t} {\mathcal A}(t) - 4 \pi.
\end{equation}
It follows that for any time $t_0 > 0$ at which the mean curvature
flow exists, we would have ${\mathcal A}(t) \le 0$ for 
$t \ge t_0 \left( 1 + \frac{{\mathcal A}(t_0)}{16 \pi t_0} \right)^4$.
Thus the mean curvature flow must go singular.
\end{proof}

\begin{remark}
The analog of Proposition \ref{ext}, in one dimension lower, is no
longer true, as can been seen by taking a closed geodesic in a 
flat $2$-torus.  This contrasts with the fact that any
compact mean curvature flow in $\R^n$ has a finite-time singularity.
\end{remark}

\subsection{Mean curvature solitons} \label{subsection5.2}

Suppose that $(M, g(\cdot), \overline{f}(\cdot))$ 
is a gradient soliton solution to 
the Ricci flow. 
We recall that this means
\begin{enumerate}
\item $(M, g(\cdot))$ is a 
Ricci flow solution,
\item At time $t$ we have
\begin{equation}
R_{\alpha \beta} + \nabla_\alpha \nabla_\beta \overline{f} =
\frac{c}{2t} g_{\alpha \beta}, 
\end{equation}
where $c = 0$ in the steady case (for $t \in \R$), $c = -1$ in the shrinking
case (for $t \in (-\infty, 0)$) and $c = 1$ in the expanding case
(for $t \in (0, \infty)$), and
\item
The function $\overline{f}$ satisfies
\begin{equation} \label{feqn}
\frac{\partial \overline{f}}{\partial t} = |\nabla \overline{f}|^2.
\end{equation}
\end{enumerate}

\begin{definition} \label{def3}
At a given time $t$, a hypersurface $\Sigma_t$ is a mean curvature soliton if 
\begin{equation} \label{5.5}
H + e_0 \overline{f} = 0.
\end{equation}
\end{definition}

Equation (\ref{5.5}) involves no choice of local orientations.

When restricted to $\Sigma_t$, the equations 
$R_{ij} + \nabla_i \nabla_j \overline{f} = 0$ and 
$R_{i0} + \nabla_i \nabla_0 \overline{f} = 0$ become

\begin{align} \label{5.6}
R_{ij} + \widehat{\nabla}_i \widehat{\nabla}_j \overline{f} +
H A_{ij} & = 0, \\
R_{i0} - \widehat{\nabla}_i H + A_i^{\: \: k} \:
\widehat{\nabla}_k \overline{f} & = 0. \notag 
\end{align}

\begin{example} \label{ex3}
If $M = \R^n$ and $g(t) = g_{\Flat}$, let $L$ be a linear
function on $\R^n$. Put $\overline{f} = L + t |\nabla L|^2$,
so that $\overline{f}$ satisfies (\ref{feqn}).
Then after changing $\overline{f}$ to $-f$, the equations in 
(\ref{5.6}) become
\begin{align}
\widehat{\nabla}_i \widehat{\nabla}_j f & =  H A_{ij}, \\
\widehat{\nabla}_i H + A_i^{\: \: k} \:
\widehat{\nabla}_k f & = 0, \notag 
\end{align}
which appear on \cite[p. 219]{Hamilton (1995)} as equations
for a translating soliton.
\end{example}

If $(M, g(\cdot), \overline{f}(\cdot))$ is a gradient steady soliton, 
let $\{\phi_t\}$ be the one-parameter family of diffeomorphisms generated
by the time-independent vector field  $- \nabla_{g(t)} \overline{f}(t)$, with
$\phi_0 = \Id$. If $\Sigma_0$ is a mean curvature soliton at time zero then its
ensuing mean curvature flow $\{\Sigma_t\}$ consists of
mean curvature solitons, and $\{\Sigma_t\}$ differs from
$\{\phi_t(\Sigma_0)\}$ by hypersurface diffeomorphisms.

There is a similar description of the mean curvature flow of
a mean curvature soliton if $(M, g(\cdot), \overline{f}(\cdot))$ 
is a gradient shrinking soliton or a gradient expanding soliton.

\begin{proposition} \label{prop6}
If $(M, g(\cdot), \overline{f}(\cdot))$ 
is a gradient steady soliton and $\{\Sigma_t\}$ is
the mean curvature flow of a mean curvature soliton then
\begin{equation} \label{5.7}
\frac{\partial H}{\partial t} -
2 \langle \widehat{\nabla} \overline{f}, \widehat{\nabla} H \rangle
+ A( \widehat{\nabla} \overline{f}, \widehat{\nabla} \overline{f})
+ 2 R^{0i} \widehat{\nabla}_i \overline{f} - \frac12 \nabla_0 R - H R_{00} = 0.
\end{equation}
\end{proposition}
\begin{proof}
We clearly have 
$\left( R_{ij} + \nabla_i \nabla_j \overline{f} \right) \Big|_{\Sigma_t} = 0$
and $\left( R_{i0} + \nabla_i \nabla_0 \overline{f} \right) 
\Big|_{\Sigma_t} = 0$. 
The proposition now follows from Proposition \ref{prop3}, along the lines of the
proof of Proposition \ref{prop4}.
\end{proof}

\begin{example} \label{ex4}
Suppose that $M = \R^n$, $g(t) = g_{\Flat}$, $L$ is a linear
function on $\R^n$ and $\overline{f} = L + t |\nabla L|^2$.
After putting $V(t) = - \widehat{\nabla} \overline{f}$, 
Proposition \ref{prop6} is the same as \cite[Lemma 3.2]{Hamilton (1995)}.
\end{example}

\subsection{Huisken monotonicity} \label{subsection5.3}

\begin{proposition} \label{prop7}
\cite{Magni-Mantegazza-Tsatis (2009)}
If $\{\Sigma_t\}$ is a mean curvature flow of compact hypersurfaces
in a gradient steady Ricci soliton
$(M, g(\cdot), \overline{f}(\cdot))$
then $\int_{\Sigma_t} e^{- \overline{f}(t)} \: dA$ is nonincreasing in
$t$. It is constant in $t$ if and only if $\{\Sigma_t\}$ are
mean curvature solitons.
\end{proposition}
\begin{proof} 1.
Using the mean curvature flow to relate nearly $\Sigma_t$'s, 
Lemma \ref{area} gives
\begin{align} \label{5.9}
\frac{d}{dt} \int_{\Sigma_t} e^{- \overline{f}(t)} \: dA = &
- \int_{\Sigma_t} \left( \frac{d \overline{f}}{dt} + R_i^{\: \: i} + H^2 \right) 
e^{- \overline{f}(t)} \: dA \\
= & - \int_{\Sigma_t} \left( \frac{\partial \overline{f}}{\partial t}
+ H e_0 \overline{f} + R_i^{\: \: i} + H^2 \right) e^{- \overline{f}(t)} \: dA
\notag \\
= & - \int_{\Sigma_t} \left( |\nabla \overline{f}|^2 
+ H e_0 \overline{f} + R_i^{\: \: i} + H^2 \right) e^{- \overline{f}(t)} \: dA.
\notag
\end{align}
From the soliton equation,
\begin{equation} \label{5.10}
0 = R_i^{\: \: i} + \nabla_i \nabla^i \overline{f} = 
R_i^{\: \: i} + \widehat{\nabla}_i \widehat{\nabla}^i \overline{f} 
+ \Gamma^i_{\: \: 0i} \nabla^0 \overline{f} = 
R_i^{\: \: i} + \widehat{\triangle} \overline{f} 
-H e_0 \overline{f}.
\end{equation}
Then
\begin{align} \label{5.11}
\frac{d}{dt} \int_{\Sigma_t} e^{- \overline{f}(t)} \: dA = &
- \int_{\Sigma_t} \left( - \widehat{\triangle} \overline{f} +
|\widehat{\nabla} \overline{f}|^2 + 
|e_0 \overline{f}|^2 + 2 H e_0 \overline{f} + H^2 \right)
e^{- \overline{f}(t)} \: dA \\
= & - \int_{\Sigma_t} \left( H + 
e_0 \overline{f} \right)^2 \: 
e^{- \overline{f}(t)} \: dA. \notag
\end{align}
The proposition follows.
\end{proof}
\begin{proof} 2. Let $\{\phi_t\}$ be the one-parameter family of
diffeomorphisms considered after Example \ref{ex3}. Put
$\widehat{g}(t) = \phi_t^* g(t)$ and $\widehat{f}(t) =
\phi_t^* \overline{f}(t)$. Then for all $t$, we have
$\widehat{g}(t) = g(0)$ and $\widehat{f}(t) = \overline{f}(0)$. Put
$\widehat{\Sigma}_t = \phi_t^{-1} (\Sigma_t)$. In terms
of $g(0)$ and $\overline{f}(0)$, the surfaces
$\widehat{\Sigma}_t$ satisfy the flow
\begin{equation} \label{5.12}
\frac{dx}{dt} = H e_0 + \nabla \overline{f}(0) = 
(H + e_0 \overline{f}(0)) e_0 + \widehat{\nabla} \overline{f}(0),
\end{equation}
which differs from the flow 
\begin{equation} \label{5.13}
\frac{dx}{dt} = (H + e_0 \overline{f}(0)) e_0
\end{equation}
by diffeomorphisms of the hypersurfaces.  The flow
(\ref{5.13}) is the negative gradient flow of the functional
$\widehat{\Sigma} \rightarrow \int_{\widehat{\Sigma}} 
e^{- \overline{f}(0)} dA$. Hence
$\int_{\widehat{\Sigma}_t} e^{- \overline{f}(0)} dA$ is
nonincreasing in $t$, and more precisely,
\begin{equation} \label{5.14}
\frac{d}{dt} \int_{\widehat{\Sigma}_t} e^{- \overline{f}(0)} dA =
- \int_{\widehat{\Sigma}_t} \left( H + e_0 \overline{f}(0) \right)^2 
e^{- \overline{f}(0)} dA.
\end{equation}
The proposition follows.
\end{proof}

\begin{remark} \label{rmk2}
There are evident analogs of Proposition \ref{prop7}, and its proofs, for
mean curvature flows in gradient shrinking Ricci solitons and
gradient expanding Ricci solitons. For the shrinking case,
where $t \in (-\infty,0)$, put $\tau = -t$. Then
$\tau^{- (n-1)/2} \int_{\Sigma_t} e^{- \overline{f}} \: dA$ is
nonincreasing in $t$.
When $M  = \R^n$, 
$g(\tau) = g_{\Flat}$ and $\overline{f}(x,\tau) = 
\frac{|x|^2}{4\tau}$, we recover Huisken monotonicity
\cite[Theorem 3.1]{Huisken (1990)}.
\end{remark}

\begin{remark} \label{rmk3}
With reference to the second
proof of Proposition \ref{prop7}, the second variation of the
functional $\Sigma \rightarrow \int_{\Sigma} e^{-f} dA$ was
derived in \cite{Bayle (2004)}; see
\cite{Fan (2008),Ho (2010)} for consequences.
The second variation formula also plays a role in
\cite[Section 4]{Colding-Minicozzi (2009)}.
\end{remark}

\begin{remark} \label{rmk4}
As a consequence of the monotonicity statement in Remark \ref{rmk2},
we can say the following about
singularity models; compare with \cite[Theorem 3.5]{Huisken (1990)}.
Suppose that $(M, g(\cdot), \overline{f}(\cdot))$
is a gradient shrinking Ricci soliton, defined for $t \in (-\infty, 0)$.
Let $\{\phi_t \}_{t \in (-\infty, 0)}$ be the corresponding
$1$-parameter family of diffeomorphisms, with
$\phi_t^* g(t) = g(-1)$.
Suppose that $\{\Sigma_t\}$ is a mean curvature flow in the Ricci soliton.
Suppose that for a sequence $\{t_i\}_{i=1}^\infty$
approaching zero from below, a smooth limit $\lim_{i \rightarrow \infty} 
\phi_{t_i}^{-1} \left( \Sigma_{t_t} \right)$ exists
and is a compact hypersurface $\Sigma_{\infty}$ in $(M, g(-1))$. Then
$\Sigma_\infty$ is a time $-1$ mean curvature soliton.
\end{remark}
\bibliographystyle{acm}

\end{document}